\documentclass[12pt]{amsart}

\usepackage{color, graphicx}
\usepackage[usenames,dvipsnames,svgnames,table]{xcolor}

\usepackage{amsmath, amsthm, stmaryrd}
\usepackage{amsfonts}
\usepackage{latexsym, amssymb}
\usepackage[T1]{fontenc}

\providecommand{\U}[1]{\protect\rule{.1in}{.1in}}

\newtheorem{theorem}{Theorem}[section]

\newtheorem{conjecture}[theorem]{Conjecture}
\newtheorem{corollary}[theorem]{Corollary}

\newtheorem{example}[theorem]{Example}

\newtheorem{lemma}[theorem]{Lemma}
\newtheorem{notation}[theorem]{Notation}

\newtheorem{prop}[theorem]{Proposition}
\newtheorem{proposition}[theorem]{Proposition}
\newtheorem{cor}[theorem]{Corollary}
\newtheorem{rem}[theorem]{Remark}

\def\N{\mathbb{N}}
\def\R{\mathbb{R}}
\def\Z{\mathbb{Z}}
\def\Q{\mathbb{Q}}
\def\C{\mathcal{C}}

\DeclareMathOperator{\pp}{Preper}
\DeclareMathOperator{\per}{Per}
\DeclareMathOperator{\num}{num}
\DeclareMathOperator{\den}{den}
\DeclareMathOperator{\supp}{supp}

\DeclareMathOperator{\id}{Id}

\title{Some results on the Flynn-Poonen-Schaefer conjecture}
\author{Shalom Eliahou and Youssef Fares}

\address{Shalom Eliahou, Univ. Littoral C\^ote d'Opale, EA 2597 - LMPA - Laboratoire de Math\'ematiques Pures et Appliqu\'ees Joseph Liouville, F-62228 Calais, France and CNRS, FR 2956, France} \email{eliahou@univ-littoral.fr}

\address{Youssef Fares, LAMFA, CNRS-UMR 7352, Université de Picardie, 80039 Amiens, France} \email{youssef.fares@u-picardie.fr}

\date{}

\begin{document}
\maketitle

\begin{abstract} For $c \in \Q$, consider the quadratic polynomial map $\varphi_c(x)=x^2-c$.  Flynn, Poonen and Schaefer conjectured in 1997 that no rational cycle of $\varphi_c$ under iteration has length more than $3$. Here we discuss this conjecture using arithmetic and combinatorial means, leading to three main results. First, we show that if $\varphi_c$ admits a rational cycle of length $n \ge 3$, then the denominator of $c$ must be divisible by $16$. We then provide an upper bound on the number of periodic rational points of $\varphi_c$ in terms of the number of distinct prime factors of the denominator of $c$. Finally, we show that the Flynn-Poonen-Schaefer conjecture holds for $\varphi_c$ if that denominator has at most two distinct prime factors. 
\end{abstract}

\section{Introduction}
Let $S$ be a set and $\varphi :S \to S$ a self map. For $z \in S$, the \emph{orbit} \emph{of $z$} \emph{under $\varphi$} is the sequence of iterates
$$O_{\varphi }(z)=(\varphi^k(z))_{k \ge 0},$$ 
where $\varphi ^k$ is the $k^{th }$ iterate of $\varphi$ and $\varphi^0=\id_S$. We say that $z$ is \emph{periodic} under $\varphi $  if there is an integer $n \ge 1$ such that $\varphi^n(z) = z$, and then the least such $n$ is the \emph{period} of $z$. In that case,  we identify $O_{\varphi }(z)$ with the finite sequence $\C= (z, \varphi(z), \dots, \varphi^{n-1}(z))$, and we say that $\C$ is a \emph{cycle} of length $n$. The element $z$ is said to be \emph{preperiodic} under $\varphi $  if there is an integer $m \ge 1$  such that $ \varphi ^m(z)$ is periodic. For every rational fraction in $\Q(x)$ of degree $\geq 2$, its set of preperiodic points is \emph{finite}, this being a particular case of a well known theorem of Northcott \cite{Nor}. However, determining the cardinality of this set is very difficult in general. The following conjecture due to Flynn, Poonen and Schaefer~\cite{FPS} illustrates the difficulty in understanding, in general, the periodic points of polynomials, even those of degree $2$.

\begin{conjecture}\label{conj FPS} Let $c \in \Q$. Consider the quadratic map $\varphi_c \colon \Q \to \Q$ defined by\footnote{The map $x \mapsto x^2+c$ is more common in the literature, but we slightly prefer to deal with $x \mapsto x^2-c$.}  $\varphi_c(x)=x^2-c$ for all $x \in \Q$. Then every periodic point of $\varphi _c$ in $\Q$ has period at most $3$.
\end{conjecture}

See also~\cite{P} for a refined conjecture on the rational \emph{preperiodic} points of quadratic maps over $\Q$. As the following classical example shows, rational points of period $3$ do occur for suitable $c \in \Q$. 
\begin{example}\label{ex length 3} Let $c = 29/16$. Then the map $\varphi_c$ admits the cycle $\C= (-1/4, -7/4,5/4)$ of  length $3$.
\end{example}

While Conjecture~\ref{conj FPS} has already been explored in several papers, it remains widely open at the time of writing. The main positive results concerning it are that period $4$ and period $5$ are indeed excluded, by Morton \cite{M} and by Flynn, Poonen and Schaefer \cite{FPS}, respectively.

\begin{theorem}[Morton]\label{thm1} For every $c\in \Q$, there is no  periodic point of $\varphi _c$ in $\Q$ of period $4$.
\end{theorem}

\begin{theorem}[Flynn, Poonen and Schaefer]\label{thm2} For every $c\in \Q$, there is no periodic point of $\varphi _c$ in $\Q$ of period $5$.
\end{theorem}

No period higher than $5$ has been excluded so far for the rational maps $\varphi_c$. However, Stoll showed that the exclusion of period $6$ would follow from the validity of the Birch and Swinnerton-Dyer conjecture \cite{St}. 

Conjecture~\ref{conj FPS} is often studied using the \emph{height} and \emph{$p$-adic Julia sets}. Here we mainly use arithmetic and combinatorial means. Among our tools, we shall use the above two results and Theorem~\ref{thm3}, a particular case of a theorem of Zieve~\cite{Z} on polynomial iteration over the $p$-adic integers.

Given $0 \not=c \in \Q$, let $s$ denote the number of distinct primes dividing the denominator of $c$.  In \cite{CG}, Call and Goldstine showed that the number of rational preperiodic points of $\varphi_c$ does not exceed the upper bound $2^{s+2}+1$. Among our present results, we show that any rational cycle of $\varphi_c$ has length at most $2^s +2$. We also show that the conjecture holds for $\varphi_c$ in case $s \le 2$.

For convenience, in order  to make this paper as self-contained as possible, we provide short proofs of some already known basic results.

\subsection{Notation}
Given $c \in \Q$, we denote by $\varphi_c \colon \Q \to \Q$ the quadratic map defined by $\varphi _c(x)=x^2-c$ for all $x \in \Q$. Most papers dealing with Conjecture~\ref{conj FPS} rather consider the map $x^2+c$. Our present choice allows statements with positive rather than negative values of $c$. For instance, with this choice, we show in \cite{EF1} that if $\varphi_c$ admits a cycle of length at least $2$, then $c\ge 1$.

The sets of rational periodic and preperiodic points of $\varphi _c$ will be denoted by $\per(\varphi _c)$  and $\pp(\varphi _c)$, respectively:
\begin{eqnarray*}
\per(\varphi _c) & = & \{x \in \Q \mid \varphi _c^n(x)=x  \,\, {\textrm{for some }} n \in \N \}, \\
\pp(\varphi _c) & = & \{x \in \Q \mid \varphi _c^m(x) \in \, \per(\varphi _c) \,\, {\textrm{for some }} m \in \N\}.
\end{eqnarray*}

\smallskip
For a nonzero integer $d$, we shall denote by $\supp(d)$ the set of prime numbers $p$ dividing $d$. For instance, supp$(45)=\{3,5\}$. If $x \in \Q$ and $p$ is a prime number, the \emph{$p$-adic valuation} $v_p(x)$ of $x$ is the unique $r \in \Z \cup \{\infty\}$ such that $x=p^r {x_1}/{x_2}$ with $x_1,x_2 \notin p\Z$ coprime integers. For $z \in \Q$, its \emph{numerator} and \emph{denominator} will be denoted by $\num(z)$ and $\den(z)$, respectively. They are the unique coprime integers such that $\den(z) \ge 1$ and $z=\num(z)/\den(z)$.

As usual, the cardinality of a finite set $E$ will be denoted by $|E|$.

\section{Basic results over $\Q$}  
 \subsection{Constraints on denominators}
The aim of this section is to show that \emph{if $\varphi_c$ has a periodic point of period at least $3$, then $\den(c)$ is divisible by $16$}. The result below first appeared in \cite{WR}.

\begin{prop}\label{walde russo} Let $c \in \Q$. If $\per(\varphi _c) \not= \emptyset$, then $\den(c)=d^2$ for some $d \in \N$, and $\den(x)=d$ for all $x \in \pp(\varphi _c)$.
\end{prop}
\begin{proof} Let $p$ be a prime dividing $\den(c)$, \emph{i.e.} such that $v_p(c)<0$. Let $x \in \Q$. 

\smallskip
\noindent
\textbf{Claim.} \emph{If $v_p(x) \not=v_p(c)/2$, then the orbit of $x$ under $\varphi_c$ is infinite}. 

\smallskip
Indeed, consider the following two cases.
\begin{enumerate}
\item If $v_p(x)<v_p(c)/2$,  then $v_p(\varphi_c(x))=v_p(x^2-c)=2v_p(x)$. Thus $v_p(\varphi_c(x)) < v_p(c) < v_p(c)/2$ since $v_p(c) <0$. It follows that $v_p(\varphi_c^n(x))=2^n v_p(x)$ for all $n \ge 1$. \smallskip
\item If $v_p(x)>v_p(c)/2$, then $v_p(\varphi_c(x))=v_p(x^2-c)=v_p(c) < v_p(c)/2$ and we are back in the preceding case. In particular, we have $v_p(\varphi_c^n(x))=2^{n-1} v_p(c)$ for all $n \ge 1$.
\end{enumerate}
In both cases, the $p$-adic valuation of $\varphi_c^n(x)$ tends to $-\infty $ for $n \to \infty$, whence the claim.

\smallskip
If now $x \in \pp(\varphi_c)$, then the claim implies $v_p(x)=v_p(c)/2$. Note that such points $x$ exist by hypothesis on $\varphi_c$. Hence $v_p(c)$ is even, and since this occurs for all primes $p$ dividing $\den(c)$, it follows that $\den(c)=d^2$ for some $d \in \N$, and that $\den(x)=d$.
\end{proof}

\medskip
Consequently, since we are only interested in rational cycles of $\varphi_c$ here, we shall only consider those $c \in \Q$ such that $\den(c) = d^2$ for some $d \in \N$. Moreover, we shall frequently consider the set $\num(\per(\varphi_c))$ of numerators of rational periodic points of $\varphi_c$. 

\begin{corollary} Let $c \in \Q$. Assume $\per(\varphi _c) \not= \emptyset$. Let $d \in \N$ such that $\den(c)=d^2$. Then
 $$
 \num(\per(\varphi_c))=d\cdot \per(\varphi_c) \text{, }  \num(\pp(\varphi_c)) = d\cdot \pp(\varphi_c).
 $$ 
\end{corollary}
\begin{proof} Directly follows from the equality $\den(\pp(\varphi_c))=\{d\}$ given by Proposition~\ref{walde russo}.
\end{proof}

\subsection{Basic remarks on periodic points}
In this section, we consider periodic points of any map $f \colon A \to A$ where $A$ is a domain.

\begin{lemma}\label{basic} Let $A$ be a commutative unitary ring and $f \colon A \to A$ a self map. Let $z_1 \in A$ be a periodic point of $f$ of period $n$, and let $\{z_1,\dots,z_n\}$ be the orbit of $z_1$. Then
\begin{equation*}
\prod_{1\leq i<j\leq n}(f(z_i)-f(z_j))=(-1)^{n-1}\prod_{1\leq i<j\le n}(z_{i}-z_{j}).
\end{equation*}
\end{lemma}
\begin{proof} We have $f(z_i)=z_{i+1}$ for all $1 \le i <n$ and $f(z_n)=z_1$. Hence
\begin{align*}
\prod_{1\leq i<j\leq n}(f(z_i)-f(z_j)) &= \prod_{1\leq i<j< n}(z_{i+1}-z_{j+1})\prod_{1\leq i< n}(z_{i+1}-z_{1}) \\
&= (-1)^{n-1}\prod_{1\leq i<j\le n}(z_{i}-z_{j}).
\end{align*}
\end{proof}

\begin{proposition}\label{prp1} Let $A$ be a domain and $f \colon A \to A$ a map of the form $f(x)=x^2-c$ for some $c \in A$. Assume that $f$ admits a cycle~in~$A$. 

\smallskip
$(i)$ Let $x,y \in A$ be two distinct periodic points of $f$, of period $m$ and $n$, respectively. Let $r= \emph{lcm}(m,n)$. Then
$\displaystyle \prod_{i=0}^{r-1}(f^i(x)+f^i(y)) \,=\, 1.$

\smallskip
$(ii)$ Assume $\per(f)=\{x_1, x_2, \ldots, x_N\}$. Then 
$\displaystyle \prod_{1\leq i<j\leq N}(x_i+x_j) \,=\, \pm 1.$
\end{proposition}
\begin{proof} First observe that for all $u,v \in A$, we have
\begin{equation}\label{der phi}
f(u)-f(v) = (u-v)(u+v).
\end{equation}

\smallskip
$(i)$ Since $f^r(x)=x$ and $f^r(y)=y$, we have
\begin{equation}\label{0 to r-1}
\prod_{i=0}^{r-1}\big(f^{i+1}(x)-f^{i+1}(y)\big) = \prod_{i=0}^{r-1}
\big(f^{i}(x)-f^{i}(y)\big).
\end{equation}
Now, it follows from \eqref{der phi} that
$$
f^{i+1}(x)-f^{i+1}(y) = (f^{i}(x)-f^{i}(y))(f^{i}(x)+f^{i}(y)).
$$
Since the right-hand side of \eqref{0 to r-1} is nonzero, the formula in $(i)$ follows.

\smallskip
Moreover, since $f$ permutes $\per(f)$, we have
$$
\prod_{1\leq i<j\leq n}(f(x_i)-f(x_j)) = \, \pm \prod_{1\leq i<j\leq n} (x_i- x_j).
$$
Using \eqref{der phi}, and since the above terms are nonzero, the formula in $(ii)$ follows.
\end{proof}

\subsection{Sums of periodic points}
Here are straightforward consequences of Proposition~\ref{prp1} for $\varphi_c$. The result below originally appeared in \cite{EF}. 

\begin{proposition}\label{egalités1,2} Let $c \in \Q$. Assume $\per(\varphi _c)=\{x_1, x_2, \ldots, x_n\}$ with $n \ge 1$. Let $d = \den(x_1)$ and \emph{$X_i=\num(x_i)$} for all $1 \le i \le n$. Then, for all $1 \le k \le n-1$, we have
\begin{eqnarray}
\prod_{1\leq i \leq n} \big(X_i+X_{i+k}\big) & = & d^n \ \text{(with indices read mod $n$)}, \label{egalite1} \\
\prod_{1\leq i<j\leq n} \big(X_i+X_j\big) & =  & \pm d^{n(n-1)/2}. \label{egalite2}
\end{eqnarray}
\end{proposition}

\begin{proof} By Proposition~\ref{walde russo}, we have $\den(x_i)=d$ for all $i$. Now chase the denominator in the formulas of Proposition~\ref{prp1}.
\end{proof}

These other consequences will play a crucial role in the sequel.

\begin{corollary}\label{lm1} Let $c \in \Q$. Let $x,y$ be two distinct points in $\per{(\varphi_c)}$. Set $X=\num(x), Y=\num(y)$ and $d = \den(x)$.
Then
\begin{enumerate}
\item[(i)] $\supp(X+Y) \subseteq \supp(d)$. That is, any prime $p$ dividing $X+Y$ also divides $d$.
\item[(ii)] $X$ and $Y$ are coprime.
\item[(iii)] If no odd prime factor of $d$ divides $X+Y$, then $X+Y= \pm 2^t$ for some $t \in \N$.	
\end{enumerate}
\end{corollary}

\begin{proof} The first point directly follows from equality~\eqref{egalite2}. For the second one, if a prime $p$ divides $X$ and $Y$, then it divides $d$ by the first point, a contradiction since $X,d$ are coprime. The last point follows from the first one and the hypothesis on the odd factors of $d$, which together imply $\supp(X+Y) \subseteq \{2\}$.
\end{proof}

\begin{example}
Consider the case $c=29/16$ of Example~\ref{ex length 3}, where $d=4$ and $\varphi_c$ admits the cycle $\C= (-1/4, -7/4,5/4)$. Here $\num(\C)=(-1,-7,5)$, with pairwise sums $-8, -2, 4$, respectively. This illustrates all three statements of Corollary~\ref{lm1}. Viewing $\C$ as a set, we have $\C \subseteq \per(\varphi_c)$. We claim $\C = \per(\varphi_c)$. For otherwise, let $x=X/4$ be yet another periodic point of $\varphi_c$. Then $X-1,X-7,X+5$ would also be powers of $2$ up to sign. The only possibility is $X=3$ as easily seen. But $3/4$ is only a \emph{preperiodic} point, since under $\varphi_c$ we have
$
3/4 \mapsto -5/4 \mapsto -1/4 \mapsto -7/4 \mapsto 5/4 \mapsto -1/4.
$
\end{example}

\subsection{Divisibility properties of $\den(c)$}
Our bounds on cycle lengths of $\varphi_c$ involve the denominator of $c$. The following proposition and corollary already appear in \cite{WR}.

\begin{prop}\label{d pair} Let $c \in \Q$. If $\den(c)$ is odd, then $|\per(\varphi_c)| \le 2$.
\end{prop}
\begin{proof} We have $\den(c)=d^2$ for some $d \in \N$, and $\den(x)=d$ for all $x \in \pp(\varphi_c)$. Assume $\per(\varphi_c)=\{x_1,\dots,x_n\}$. Let $X_i=\num(x_i)$ for all $i$. Then by equality \eqref{egalite2} in Proposition~\ref{egalités1,2}, we have
$$\prod_{1\leq i<j\leq n} \big(X_i+X_j\big)  =  \pm d^{n(n-1)/2}.$$
Since $d$ is odd by assumption, each factor $X_i+X_{j}$ is odd as well, whence $X_i \not\equiv X_j \bmod 2$ for all $1 \le i < j \le n$. Of course, this is only possible if $n \le 2$.
 \end{proof}
 
 \begin{rem}
If $c \in \Z$, then $\den(c)=1$ and the above result implies that $\varphi_c$  admits at most two periodic points. 
 \end{rem}

\begin{cor}[\cite{WR}]\label{cor d pair} Let $c \in \Q$. If $\varphi _c$ admits a rational cycle of length at least $3$, then $\den(c)$ is even.
\end{cor}

\subsection{Involving $p$-adic numbers}

We shall now improve Corollary~\ref{cor d pair} by showing that under the same hypotheses, $\den(c)$ must in fact be divisible by $16$. For that, we shall need Morton's Theorem~\ref{thm1} excluding period $4$, as well as a result below due to Zieve concerning periodic points of polynomials over the $p$-adic integers. 

As usual, $\Z_p$ and $\Q_p$ will denote the rings of $p$-adic integers and  numbers, respectively. A result in \cite{AF} contains a generalization of the above proposition. It 
says that any polynomial $g(x)=x^p+ \alpha$ with $ \alpha \in \Z_p$, either admits $p$  fixed points in $\Q_p$ or else a cycle of length exactly $p$ in $ \Q_p$. For $z \in \Q_p$, we denote by $v_p(z)$ the $p$-adic valuation of $z$. 

Here is a particular case of a theorem of Zieve~\cite{Z} that we shall use to improve Corollary~\ref{cor d pair}. See also~\cite[Theorem 2.21 p. 62]{S}.
\begin{theorem}\label{thm3} Let $p$ be a prime number and let $g$ be a polynomial in $\Z_p[t]$ of degree at least $2$.  Let $\alpha$ be a periodic point of $g$ in $\Z_p$   and let
\begin{eqnarray*}  
n & = & \textrm{the exact period of $\alpha$ in $\Z_p$}, \\
m & = & \textrm{the exact period  of $\alpha$ in $\Z/ p\Z$}, \\
r & = &  \left\{
  \begin{array}{ll}
  \textrm{the order of  } (g^m)'(\alpha) & \textrm{if } \, (g^m)'(\alpha) \,  \textrm{is invertible in } \Z/ p\Z, \\ & \\
  \infty & \textrm{if } \, (g^m)'(\alpha) \,  \textrm{is not invertible in } \Z/ p\Z.
  \end{array}
  \right.
\end{eqnarray*} 
Then $n \in \{m, mr, mrp^e\}$ for some integer $e \ge 1$ such that $p^{e-1}\le 2/(p-1).$
\end{theorem}

We may now sharpen Corollary~\ref{cor d pair}.

 \begin{theorem}\label{4|d} \label{div by 16} Let $c \in \Q$. If $\varphi _c$ admits a rational cycle of length $n \ge 3$, then $\den(c)$ is divisible by $16$. 
\end{theorem}
\begin{proof} By Propositions~\ref{walde russo} and~\ref{d pair}, we have $\den(c)=d^2$ for some even positive integer $d$. Assume for a contradiction that $d$ is not divisible by $4$. Hence $v_2(d)=1$ and $v_2(c)=-2$. Let $\C \subseteq \per(\varphi_c)$ be a rational cycle of $\varphi_c$ of length $n \ge 3$. For all $z \in \C$, we have $\den(z)=d$ and hence $v_2(z)=-1$ by Proposition~\ref{walde russo}. 

Recall that, if $z_1,z_2 \in \Q$ satisfy $v_2(z)=v_2(z')=r$ for some $r \in \Z$, then $v_2(z \pm z') \ge r+1$. 

In particular, for all $z \in \C$, we have $v_2(z-1/2) \ge 0$. Therefore the translate $\C-1/2$ of $\C$ may be viewed as a subset of the local ring $\Z_{(2)} \subset \Q$, and hence of the ring  $\Z_2$ of $2$-adic integers. That is, we have
$$
\C-1/2 \ \subset \ \Z_2.
$$
\textbf{Step 1}. In view of applying Theorem~\ref{thm3}, we seek a polynomial in $\Z_2[t]$ admitting $\C-1/2$ as a cycle. The polynomial
\begin{eqnarray*}
f(t) & = & \varphi_c(t+1/2)-1/2 \\
& = & t^2+t-(c+1/4)
\end{eqnarray*}
will do. Indeed, by construction we have
$$
f(t-1/2)=\varphi_c(t)-1/2.
$$
Since $\varphi_c(\C)=\C$, it follows that
$$
f(\C-1/2) = \C-1/2,
$$
as desired. For the constant coefficient of $f$, we claim that $v_2(c+1/4) \ge 0$. Indeed, let $x,y \in \C$ with $y = \varphi_c(x)$. Thus $f(x-1/2)=y-1/2$, i.e.
$$
(x-1/2)^2+(x-1/2)-(c+1/4) = y-1/2.
$$
Since $v_2(x-1/2), v_2(y-1/2) \ge 0$, it follows that $v_2(c+1/4) \ge 0$, as claimed. Therefore $f(t) \in \Z_2[t]$, as desired. 

For the next step, we set 
$$\C-1/2=(z_1,\dots,z_n)$$ 
with $f(z_i)=z_{i+1}$ for $i \le n-1$ and $f(z_n)=z_1$.

\medskip
\noindent
\textbf{Step 2}. We now apply Theorem \ref{thm3} to the polynomial $g=f$ and to its $n$-periodic point $\alpha=z_1$. We need to compute the corresponding numbers $m$ and $r$ in that theorem, where $m$ is the period of $z_1$ in $\Z/2\Z$. 

\smallskip
We claim that $m=1$. By Lemma~\ref{basic}, for the cycle $(z_1, z_2,\dots, z_n)$ of $f$, we have
 $$\prod_{1\le i <j \le n }^{n}\frac{f(z_i)-f(z_j)}{z_i-z_j} \,=\, \pm 1.
 $$
 Since $f(x)-f(y)=(x-y)(x+y+1)$ for all $x,y$, this yields 
 $$\prod_{1\le i <j \le n }^{n} (z_i+z_j +1) \,=\, \pm 1.
 $$
Therefore $v_2(z_i+z_j +1)=0$ for all $1\le i < j \le n$, which in turn implies $v_2(z_i-z_j) \ge 1$ for all $i<j$. Consequently, the cycle $(z_1, z_2,\dots, z_n)$ collapses to the cycle $(z_1)$ of length $1$ in $\Z/2\Z$. This settles the claim. 

\smallskip
Since $m=1$, we have $(f^m)'(t)=f'(t)=2t+1$ in $\Z_2[t]$, whence $f'(z_1)=1$ in $\Z/2\Z$. Therefore $r=1$ by definition. 

\smallskip
 By Theorem \ref{thm3}, it follows that $n \in \{1,2^e\}$ for some integer $e \ge 1$ such that $2^{e-1} \le 2/1$. Hence $e \le 2$ and so $n \in \{1,2,4\}$. Since $n \ge 3$ by assumption, it follows that $n=4$. But period 4 for $\varphi_c$ is excluded by Morton's Theorem~\ref{thm1}. This contradiction concludes the proof of the theorem.
 \end{proof}

\begin{rem} Theorem~\ref{div by 16} is best possible, as witnessed by Example~\ref{ex length 3} where period 3 occurs for $\varphi_c$ with $c=29/16$.
\end{rem}

\section{An upper bound on $|\per(\varphi_c)|$}

Let $c \in \Q$. Throughout this section, we assume $\den(c) = d^2$ with $d \in 4\N$. Recall that this is satisfied whenever $\varphi_c$ admits a rational cycle $\C$ of length $n \ge 3$, as shown by Proposition~\ref{walde russo} and Theorem~\ref{4|d}.

Let $s=|\supp(d)|$. The following upper bound on $|\pp(\varphi _c)|$ was shown in \cite{CG}:
$$
|\pp(\varphi _c)|\leq 2^{s+2} +1.
$$
Our aim in this section is to obtain an analogous upper bound on $|\per(\varphi _c)|$, namely
$$
|\per(\varphi _c)|\leq 2^{s} +2.
$$

The proof will follow from a string of modular constraints on the numerators of periodic points of $\varphi _c$ developped in this section. 

\subsection{Constraints on numerators}
We start with an easy observation. 

\begin{lemma}\label{num preper} Let $c = a/d^2 \in \Q$ with $a,d$ coprime integers. Let $x \in \pp(\varphi_c)$. Let $X=\num(x)$. Then $X^2 \equiv a \bmod d$.
\end{lemma}
\begin{proof} We have $x=X/d$ by Proposition~\ref{walde russo}. Let $z = \varphi_c(x)$. Then $z \in \pp(\varphi_c)$, whence $z = Z/d$ where $Z=\num(z)$. Now $z=x^2-c=(X^2-a)/d^2$, whence
\begin{equation}\label{x'}
Z=(X^2-a)/d.
\end{equation}
Since $Z$ is an integer, it follows that
$
X^2 \equiv a \bmod d.
$
\end{proof}

Here is a straightforward consequence.

\begin{proposition}\label{lm2} Let $ c \in \Q$ such that $\den(c)=d^2$ with $d \in 4\N$. Let $X,Y \in \num(\pp(\varphi_c))$. Let $p \in \supp(d)$ and $r = v_p(d)$ the $p$-valuation of $d$.
Then
$$
X \equiv \pm Y \bmod p^r.
$$
In particular, $\num(\pp(\varphi_c))$ reduces to at most two opposite classes mod~$p^r$.
\end{proposition}

\begin{proof} It follows from Lemma~\ref{num preper} that
$X^2 \equiv Y^2 \bmod d.$
Hence
$$
(X+Y)(X-Y) \equiv 0 \bmod p^r.
$$

\noindent
\textbf{Case 1.} Assume $p$ is odd. Then $p$ cannot divide both $X+Y$ and $X-Y$, for otherwise it would divide $X$ which is impossible since $X$ is coprime to $d$. Therefore $p^r$ divides $X+Y$ or $X-Y$, as desired.

\medskip
\noindent
\textbf{Case 2.} Assume $p=2$. Then $r \ge 2$ by hypothesis.  Let $x'=\varphi_c(x)=X'/d$ and $y'=\varphi_c(y)=Y'/d$. Then $X',Y'$ are odd since coprime to $d$. By \eqref{x'}, we have $X'=(X^2-a)/d$ and $Y'=(Y^2-a)/d$. Hence 
$$
X'-Y'=(X^2-Y^2)/d.
$$
Since $2^r$ divides $d$ and since $X'-Y'$ is even, it follows that
$$
(X+Y)(X-Y) \equiv 0 \bmod 2^{r+1}.
$$
Now $4$ cannot divide both $X+Y$ and $X-Y$ since $X,Y$ are odd. Therefore
$X+Y \equiv 0 \bmod 2^r$ or $X-Y \equiv 0 \bmod 2^r$, as desired.
\end{proof}

\begin{corollary}\label{preper mod d} Let $ c \in \Q$ such that $\den(c)=d^2$ with $d \in 4\N$. Let $s = |\supp(d)|$. Then $\num(\pp(\varphi_c))$ reduces to at most $2^s$ classes mod~$d$.
\end{corollary}

\begin{proof} Set $\supp(d)=\{p_1, \dots, p_s\}$ and $d=p_1^{r_1}\dots p_s^{r_s}$. 
By Proposition~\ref{lm2}, the set $\num(\pp(\varphi_c))$ covers at most 2 distinct classes mod $p_i^{r_i}$ for all $1 \le i \le s$. Therefore, by the Chinese Remainder Theorem, this set covers at most $2^s$ distinct classes mod $d$.
\end{proof}

The particular case in Proposition~\ref{lm2} where $X,Y \in \num(\per(\varphi_c))$ and $X \equiv +Y \bmod p^r$ for all $p \in \supp(d)$, i.e. where $X \equiv Y \bmod d$, has a somewhat surprising consequence and will be used more than once in the sequel.

\begin{proposition}\label{X equiv Y} Let $ c \in \Q$ such that $\den(c)=d^2$ with $d \in 4\N$. Let $X,Y \in \num(\per(\varphi_c))$ be distinct. If $X \equiv Y \bmod d$, then $X+Y=\pm 2$.
\end{proposition}
\begin{proof} As $X,Y$ are coprime to $d$, they are odd. We claim that $\supp(X+Y)=\{2\}$. Indeed, let $p$ be any prime factor of $X+Y$. Then $p$ divides $d$ by Corollary~\ref{lm1}. Hence $p$ divides $X-Y$ since $d$ divides $X-Y$ by hypothesis. Therefore $p$ divides $2X$, whence $p=2$ since $X$ is odd. It follows that $X+Y = \pm 2^t$ for some integer $t \ge 1$. Since $d \in 4\N$ and $d$ divides $X-Y$, it follows that $4$ divides $X-Y$. Hence $4$ cannot also divide $X+Y$ since $X,Y$ are odd. Therefore $t=1$, i.e. $X+Y = \pm 2$ as desired.
\end{proof}

\begin{example} Consider the case $c=29/16$ of Example~\ref{ex length 3}, where $\varphi_c$ admits the cycle $\C=(-1/4,-7/4,5/4)$. In $\num(\C)=(-1,-7,5)$, only $-7$ and $5$ belong to the same class mod $4$, and their sum is $-2$ as expected. 
\end{example}

\subsection{From $\Z/d\Z$ to $\Z$}
Our objective now is to derive from Proposition~\ref{lm2} the upper bound $|\per(\varphi_c)| \le 2^s+2$ announced earlier. For that, we shall need the following two auxiliary results.

\begin{lemma}\label{lm3}Let  $k\in \N$. Up to order, there are only two ways to express $2^k$ as $2^k=\varepsilon _12^{k_1}+\varepsilon _22^{k_2}$
with $\varepsilon _1,\varepsilon _2=\pm 1$ and $k_1,k_2\in \N$.
\end{lemma}
\begin{proof} We may assume $k_1 \le k_2$. There are two cases. 
\begin{enumerate}
\item If $k_1=k_2$, then $2^{k_1}(\varepsilon _1+\varepsilon _2)=2^k$, implying $k_1=k_2=k-1$ and $\varepsilon _1=\varepsilon _2=1$.
\item If $k_1 < k_2$, then $2^{k_1}(\varepsilon _1+\varepsilon _2 2^{k_2-k_1})=2^k$, implying $k=k_1=k_2-1$, $\varepsilon _1=-1$ and $\varepsilon _2=1$. \qedhere
\end{enumerate}
\end{proof}

\bigskip
\begin{proposition}\label{lm4} Let $ c \in \Q$ such that $\den(c)=d^2$ with $d \in 4\N$. If there are distinct pairs $\{X_1,Y_1\}, \{X_2,Y_2\} \subseteq \num(\per(\varphi _c))$ such that $X_1+Y_1 = \pm (X_2+Y_2) = \pm 2^k$ for some $k \in \N$, then
$$
X_1+Y_1 = -(X_2+Y_2).
$$
\end{proposition}

\begin{proof}
 Assume for a contradiction that $X_1+Y_1 = X_2+Y_2 = \pm 2^k$. Let $p \in \supp(d)$ be odd. We claim that $X_1, X_2, Y_1,Y_2$ \emph{all belong to the same nonzero class mod $p$}. Indeed, we know by Proposition~\ref{lm2} that $X_1, X_2, Y_1,Y_2$ belong to \emph{at most two opposite classes} mod $p$. Since $p$ does not divide $X_i+Y_i$ for $1 \le i \le 2$, i.e. $X_i \not\equiv -Y_i \bmod p$, it follows that $X_i \equiv Y_i \bmod p$. Since $X_1 \equiv \pm X_2 \bmod p$ and $X_1+Y_1 = X_2+Y_2$, it follows that $X_1 \equiv X_2 \bmod p$ and the claim is proved. Therefore no sum of two elements in $\{X_1,Y_1,X_2, Y_2\}$ is divisible by $p$. Hence, by the third point of Corollary~\ref{lm1}, any sum of two distinct elements in $\{X_1,Y_1,X_2, Y_2\}$ is equal up to sign to a power of $2$. Moreover, we have
\begin{eqnarray*}
 \pm 2^{k+1} & = & (X_1+Y_1)+(X_2+Y_2)  \\
& = & (X_1+X_2)+(Y_1+Y_2) \\
& = & (X_1+Y_2)+(X_2+Y_1).
\end{eqnarray*}
It now follows from Lemma \ref{lm3} that at least two of $X_1, Y_1, X_2,Y_2$ are equal. This contradiction concludes the proof.
\end{proof}

\begin{notation}
For any $h \in \Z$, we shall denote by $\pi_h \colon \Z \to \Z/h\Z$ the canonical quotient map mod $h$.
\end{notation}

\begin{theorem}\label{thm 4} Let $ c \in \Q$ such that $\den(c)=d^2$ with $d \in 4\N$. Let $m = |\pi_d(\num(\per(\varphi _c)))|$. 
Then $$m \le |\per(\varphi_c)|\,\le\, m+2.$$
\end{theorem}

\begin{proof} The first inequality is obvious. We now show  $|\per(\varphi_c)| \le m+2$.

\smallskip
\noindent
\textbf{Claim.} \emph{Each class mod $d$ contains at most 2 elements of $\num(\per(\varphi _c))$.} 

\smallskip
Assume the contrary. Then there are three distinct elements $X, Y, Z$ in $\num(\per(\varphi _c))$ such that $X \equiv Y \equiv Z \bmod d$. By Proposition~\ref{X equiv Y}, all three sums $X+Y$, $X+Z$ and $Y+Z$ belong to $\{\pm 2\}$. Hence two of them coincide, e.g. $X+Y=X+Z$. Therefore $Y=Z$, a contradiction. This proves the claim.

\smallskip
Now, assume for a contradiction that $|\per(\varphi_c)| \ge m+3$. The claim then implies that there are at least 3 distinct classes mod $d$ each containing two distinct elements in $\num(\per(\varphi _c))$. That is, there are six distinct elements $X_1, Y_1$, $X_2,Y_2$ and $X_3, Y_3$ in $\num(\per(\varphi _c))$ such that $X_i\equiv Y_i \bmod d$ for $1 \le i \le 3$. Again, Proposition~\ref{X equiv Y} implies $X_i+Y_i=\pm 2$ for $1 \le i \le 3$. This situation is excluded by Proposition~\ref{lm4}, and the proof is complete.
\end{proof}

\begin{rem}\label{m+2}
The above proof shows that if $|\per(\varphi_c)|= m+2$, then there are exactly two classes mod $d$ containing more than one element of $\num(\per(\varphi_c))$, and both classes contain exactly two such elements. Denoting $\{X_1,Y_1\}, \{X_2,Y_2\} \subset \num(\per(\varphi_c))$ these two special pairs, the proof further shows that $X_1+Y_1=\pm 2=-(X_2+Y_2)$.
\end{rem}

\begin{cor}\label{Cor} Let $ c \in \Q$ such that $\den(c)=d^2$ with $d \in 4\N$. Let $s = |\supp(d)|$. Then
$$
|\per(\varphi_c)| \,\le\, 2^s+2.
$$
\end{cor}
\begin{proof} We have $|\per(\varphi_c)|\,\le\, m+2$ by the above theorem, and $m \le 2^s$ by Corollary~\ref{preper mod d}.
\end{proof}

\subsection{Numerator dynamics}

Let $c = a/d^2 \in \Q$ with $a,d$ coprime integers. Closely related to the map $\varphi_c$ is the map $d^{-1}\varphi_a \colon \Q \to \Q$. By definition, this map satisfies
$$
d^{-1}\varphi_{a}(x) = (x^2-a)/d
$$
for all $x \in \Q$. As was already implicit earlier, we now show that cycles of $\varphi_c$ in $\Q$ give rise, by taking numerators, to cycles of $d^{-1}\varphi_{a}$ in $\Z$.

\begin{lemma}\label{num dyn} Let $c = a/d^2 \in \Q$ with $a,d$ coprime integers. Let $\C \subset \Q$ be a cycle of $\varphi_c$. Then $\num(\C) \subset \Z$ is a cycle of $d^{-1}\varphi_{a}$ of length $|\C|$.
\end{lemma}
\begin{proof} Recall that $\den(\C)=\{d\}$ by Proposition~\ref{walde russo}. Let $x \in \C$ and $y = \varphi_c(x)$. Let $X=\num(x), Y= \num(y)$. Then $x=X/d, Y=y/d$. We have $y=x^2-c=(X^2-a)/d^2$. Hence $Y=(X^2-a)/d=d^{-1}\varphi_{a}(X)$. In particular, we have the formula
\begin{equation}\label{commut diagram}
(d^{-1}\varphi_{a})(X) = d\varphi_c(X/d)
\end{equation}
for all $X \in \num(\C)$. 
\end{proof}

\subsection{The cases $d \not \equiv 0$ mod 3 or mod 5}

\begin{lemma}\label{lm5} Let $c \in \Q$ and $\C \subseteq \per(\varphi_c)$ a cycle of positive length $n$.

\vspace{1mm}
\emph{(i)} If $d \not \equiv 0 \bmod 3$ and $n \ge 3$, then $\num(\C)$ reduces mod $3$ to exactly one nonzero element.  

\vspace{1mm}
\emph{(ii)} If $d \not \equiv 0 \bmod 5$  and $n \ge 4$, then $\num(\C)$ reduces mod $5$ to exactly one or two nonzero elements mod $5$.
\end{lemma}

\begin{proof} First some preliminaries. Of course $\varphi_{c}$ induces a cyclic permutation of $\C$. By Proposition~\ref{walde russo}, we have $c= a/d^2$ with $a,d$ coprime integers. By Lemma~\ref{num dyn}, the rational map $d^{-1}\varphi_{a}$ induces a cyclic permutation of $\num(\C)$, say
$$
d^{-1}\varphi_{a} \colon \num(\C) \to \num(\C).
$$
Let $X,Y \in \num(\C)$ be distinct. Then $\supp(X+Y) \subseteq \supp(d)$ by Corollary~\ref{lm1}. In particular, let $q$ be any prime number such that $d \not\equiv 0 \bmod q$. Then
\begin{equation}\label{nonzero}
X+Y \not\equiv 0 \bmod q.
\end{equation}
Since $d$ is invertible mod $q$, we may consider the reduced map
\begin{equation}\label{mod q}
f = \pi_q \circ (d^{-1}\varphi_{a}) \colon \Z/q\Z \to \Z/q\Z,
\end{equation}
where $f(x)=d^{-1}(x^2-a)$ for all $x \in \Z/q\Z$. Thus, we may view $\pi_q(\num(\C))$ as a sequence of length $n$ in $\Z/q\Z$, where each element is cyclically mapped to the next by $f$. Note that \eqref{nonzero} implies that this $n$-sequence \emph{does not contain opposite elements $u,-u$ of $\Z/q\Z$}, and in particular 
contains \emph{at most one} occurrence of $0$.

\smallskip
We are now ready to prove statements (i) and (ii). 

\smallskip
(i) Assume $d \not\equiv 0 \bmod q$ where $q=3$. By the above, the $n$-sequence $\pi_3(\num(\C))$ consists of at most one $0$ and all other elements equal to some $u \in \{\pm 1\}$. Since $n \ge 3$, this $n$-sequence contains two cyclically consecutive occurrences of $u$. Therefore $f(u) = u$. Hence $\pi_3(\num(\C))$ contains $u$ as it unique element repeated $n$ times.

\smallskip
(ii) Assume $d \not\equiv 0 \bmod q$ where $q=5$. Since $n \ge 4$ and the $n$-sequence $\pi_5(\num(\C))$ contains at most one $0$, it must contains three cyclically consecutive nonzero elements $u_1,u_2,u_3 \in \Z/5\Z \setminus \{0\}$. Since that set contains at most two pairwise non-opposite elements, it follows that $u_i=u_j$ for some $1 \le i < j \le 3$. Now $u_1 \mapsto u_2 \mapsto u_3$ by $f$. Therefore, if either $u_1=u_2$ or $u_2=u_3$, it follows that the whole sequence $\pi_5(\num(\C))$ consists of the one single element $u_2$ repeated $n$ times. On the other hand, if $u_1\not=u_2$, then $u_1=u_3$. In this case, the $n$-sequence $\pi_5(\num(\C))$ consists of the sequence $u_1,u_2$ repeated $n/2$ times. This concludes the proof.
\end{proof}

\begin{example} Consider the case $c=a/d^2=29/16$ of Example~\ref{ex length 3}, where $\varphi_c$ admits the cycle $\C= (-1/4, -7/4,5/4)$. Then $\num(\C)=(-1,-7,5)$, a cycle of length $3$ of the map $d^{-1}\varphi_{a}=4^{-1}\varphi_{29}$. That cycle reduces mod $3$ to $(-1,-1,-1)$, as expected with statement \emph{(i)} of the lemma. Statement \emph{(ii)} does not apply since $n=3$, and it would fail anyway since $\num(\C)$ reduces mod $5$ to the sequence $(-1,-2,0)$.
\end{example}

\subsection{Main consequences}

\begin{proposition}\label{cor1}
Let $c=a/d^2 \in \Q$ with $a,d$ coprime integers and with $d \in 4\N$. Assume $d \not\equiv 0 \bmod 3$. Let $s=|\supp(d)|$. For every rational cycle $\C$ of $\varphi_c$, we have
$$
|\C| \,\le\, 2^s+1.$$
\end{proposition}

\begin{proof}
 By Corollary \ref{Cor}, we have $|\C| \,\le\, 2^s+2$. If $|\C| = 2^s+2$ then, by Remark \ref{m+2}, there exist two pairs $\{X_1, Y_1\}, \{X_2, Y_2\}$ in $\num(\C)$ such that $X_1+Y_1=2$ and $X_2+Y_2=-2$. Since $d \not\equiv 0 \bmod 3$, Lemma~\ref{lm5} implies that $X_1, X_2, Y_1,Y_2$ reduce to the same nonzero element $u$ mod $3$. This contradicts the equality $X_1+Y_1=-(X_2+Y_2)$.
\end{proof}

\begin{theorem}\label{thm 5 bis} If $\den(c)$ admits at most two distinct prime factors, then $\varphi_c$ satisfies the Flynn-Poonen-Schaefer conjecture.
\end{theorem}

\begin{proof}
Let $\C$ be a rational cycle of $\varphi _c$ of length $n \ge 3$. Then $d$ is even and hence $s \geq 1$. 

\smallskip
{\tiny $\bullet$} If $s=1$, then  $d$ is a power of  $2$. By Corollary \ref{cor1}, we have $ |\per(\varphi _c)| \, \leq 2^1+1=3$ and $|\C| \le 3$. See also \cite{EF}.

\smallskip
{\tiny $\bullet$} Assume now $s=2$. Then $d=2^{2r_1}p^{r_2}$ where $p$ is an odd prime. By Theorem~\ref{thm 4},  we have $|\C| \le |\per(\varphi _c)| \le 6$. By Theorems \ref{thm1} and \ref{thm2}, we have $|\C| \not= 4,5$. It remains to show $|\C| \not= 6$. We distinguish two cases. If $p \not= 3$, then $|\C| \leq 2^2+1=5$ by Corollary \ref{cor1} and we are done. Assume now $p=3$, so that $d=2^{2r_1}3^{r_2}$. Let $m$ denote the number of classes of $\num(\C)$ mod $q=5$. It follows from Lemma~\ref{lm5} that $m \le 2$. Since the order of every element in $(\Z/ 5\Z)^*$ belongs to $\{1, 2, 4\}$, it follows from Zieve's Theorem~\ref{thm3} that $|\C|$ is a power of $2$. Hence $|\C| \in \{1, 2, 4\}$ and we are done.
\end{proof}


\end{document}